\documentclass[12pt]{amsart}
\usepackage{a4wide,enumerate,xcolor,graphicx}
\usepackage{amsmath}
\usepackage{scrextend} 
\allowdisplaybreaks
\usepackage{esint} 

\let\pa\partial
\let\na\nabla
\let\eps\varepsilon
\newcommand{\N}{{\mathbb N}}
\newcommand{\R}{{\mathbb R}}

\newcommand{\diver}{\operatorname{div}}

\renewcommand{\H}{\mathcal{H}}

\newtheorem{theorem}{Theorem}
\newtheorem{lemma}[theorem]{Lemma}

\newtheorem{proposition}[theorem]{Proposition}
\newtheorem{remark}[theorem]{Remark}

\begin{document}

\title[Shigesada--Kawasaki--Teramoto system beyond detailed balance]{
The Shigesada--Kawasaki--Teramoto \\
cross-diffusion system beyond detailed balance}

\author[X. Chen]{Xiuqing Chen}
\address{School of Mathematics (Zuhai), Sun Yat-Sen University, Zhuhai 519082,
Guang\-dong Province, China}
\email{chenxiuqing@mail.sysu.edu.cn}

\author[A. J\"ungel]{Ansgar J\"ungel}
\address{Institute of Analysis and Scientific Computing, Technische Universit\"at Wien,
Wiedner Hauptstra\ss e 8--10, 1040 Wien, Austria}
\email{juengel@tuwien.ac.at}

\author[L. Wang]{Lei Wang}
\address{School of Mathematics (Zuhai), Sun Yat-Sen University, Zhuhai 519082,
Guang\-dong Province, China}
\email{wanglei33@mail2.sysu.edu.cn}

\date{\today}

\thanks{The first and third authors acknowledge support from the National Natural 
Science Foundation of China (NSFC), grant 11971072.
The second author acknowledges partial support from
the Austrian Science Fund (FWF), grants P33010, W1245, and F65.
This work has received funding from the European
Research Council (ERC) under the European Union's Horizon 2020 research and
innovation programme, ERC Advanced Grant no.~101018153.}

\begin{abstract}
The existence of global weak solutions to the cross-diffusion model of
Shigesada, Kawasaki, and Teramoto for an arbitrary number of species
is proved. The model consists of strongly coupled parabolic equations for the
population densities in a bounded domain with no-flux boundary conditions,
and it describes the dynamics of the segregation of the population species.
The diffusion matrix is neither symmetric nor positive semidefinite.
A new logarithmic entropy allows for
an improved condition on the coefficients of heavily nonsymmetric diffusion
matrices, without imposing the detailed-balance condition
that is often assumed in the literature.
Furthermore, the large-time convergence of the solutions
to the constant steady state is proved by using the relative entropy associated
to the logarithmic entropy.
\end{abstract}



\maketitle


\section{Introduction}

The Shigesada--Kawasaki--Teramoto (SKT) system was introduced in \cite{SKT79}
to describe the dynamics of two competing population species. In this model,
the diffusion rate of each species depends on the gradients of the densities of
both species, expressed by cross-diffusion terms. They give rise to a repulsive
effect leading to spatial segregation. The stationary model has been extended
to three species in \cite{LMN00}, while the time-dependent system for an
arbitrary number of species was investigated in \cite{WeFu09}.
The existence of global weak solutions to the transient model has been proved
only under detailed balance or imposing bounds
on the self-diffusion coefficients \cite{CDJ18}; see below for details.
In this paper, we suggest a new condition on the self-diffusion coefficients,
which is significantly weaker than that one in \cite{CDJ18} in the case
of heavily nonsymmetric diffusion matrices.

The SKT model consists of the following cross-diffusion equations for the
population densities $u_i$:
\begin{equation}\label{1.eq}
  \pa_t u_i = \diver\bigg(\sum_{j=1}^n A_{ij}(u)\na u_j\bigg), \quad
	A_{ij}(u) = \delta_{ij}a_{i0} + \delta_{ij}\sum_{k=1}^n a_{ik}u_k + a_{ij}u_i,
\end{equation}
in a bounded domain $\Omega\subset\R^d$ ($d\le 3$) for $t>0$, where $i,j=1,\ldots,n$,
and $\delta_{ij}$ is the Kronecker symbol,
supplemented by the initial and no-flux boundary conditions
\begin{equation}\label{1.bic}
  u_i(0)=u_i^0\quad\mbox{in }\Omega, \quad\sum_{j=1}^n A_{ij}(u)\na u_j\cdot\nu=0
	\quad\mbox{on }\pa\Omega,\ t>0,\ i=1,\ldots,n.
\end{equation}
The diffusion coefficients $a_{ij}$ are nonnegative numbers.
We call $a_{ii}$ the self-diffusion coefficients and $a_{ij}$ for $i\neq j$
the cross-diffusion coefficients. The original model for $n=2$ species in \cite{SKT79}
also contains a drift term involving the environmental potential and
Lotka--Volterra reaction terms.
We have neglected these terms to simplify the presentation.
Our technique is able to treat these terms; see, e.g., \cite{ChJu04,ChJu06}.

While the global existence analysis in the two-species model
is quite well understood \cite{ChJu04,ChJu06,Dre08},
the global existence of weak solutions to the $n$-species system
has been proven only in the so-called detailed-balance case (see below)
\cite{CDJ18} and in the case of large self-diffusion coefficients; see, e.g.,
\cite{DLM14,DLMT15,LeMo17}. Another approach was suggested by Amann \cite{Ama89},
who proved that a priori estimates in the $W^{1,p}(\Omega)$ norm with $p>d$ are
sufficient for the solutions to general quasilinear parabolic
systems to exist globally in time, and he applied his result
to the triangular case, where $a_{ij}=0$ for $i>j$.
However, $W^{1,p}(\Omega)$ estimates with $p>d$ for
solutions to \eqref{1.eq} under general conditions seem to be out of reach.

The main difficulty in the analysis of \eqref{1.eq}--\eqref{1.bic} is the
fact that the diffusion matrix is generally neither symmetric nor positive
semidefinite. This issue was overcome in \cite{CDJ18} by exploiting the
entropy structure of \eqref{1.eq}. This means that there exists a so-called
entropy density $h:[0,\infty)\to\R$ such that \eqref{1.eq} can be written in
terms of the entropy variables $w_i=\pa h/\pa u_i$ as
$$
  \pa_t u_i(w) = \diver\bigg(\sum_{j=1}^n B_{ij}(w)\na w_j\bigg), \quad i=1,\ldots,n,
$$
where $u_i$ is interpreted as a function of $w$ and
$B(w)=A(u(w))h''(u(w))^{-1}$ with $B(w)=(B_{ij}(w))\in\R^{n\times n}$
is a positive semidefinite matrix, called the mobility matrix.
Here, $w=(w_1,\ldots,w_n)$ and
$u=(u_1,\ldots,u_n)$ are vector-valued functions.
For instance, we introduce the entropy density
$$
  \widetilde{h}(u)=\sum_{i=1}^n\pi_i u_i(\log u_i-1)dx,
$$
where $\pi_i>0$ for $i=1,\ldots,n$ are assumed to satisfy $\pi_i a_{ij}=\pi_j a_{ji}$
for all $i,j=1,\ldots,n$. These equations constitute the detailed-balance condition
for the Markov chain associated to $(a_{ij})$, and $(\pi_1,\ldots,\pi_n)$ is
the corresponding invariant measure. Under this assumption,
a formal computation shows that,
along solutions to \eqref{1.eq}--\eqref{1.bic},
$$
  \frac{d}{dt}\int_\Omega \widetilde{h}(u(t))dx
	+ 4\sum_{i=1}^n\pi_i a_{i0}\int_\Omega|\na\sqrt{u_i}|^2 dx
	+ 2\sum_{i=1}^n\pi_i a_{ii}\int_\Omega|\na u_i|^2 dx \le 0,
$$
which provides suitable gradient estimates. It was shown in \cite{CDJ18}
that the detailed-balance condition is not necessary for a global existence
analysis. If self-diffusion dominates cross-diffusion in the sense
\begin{equation}\label{1.wcd}
  4a_{ii} > \sum_{j=1}^n
	\big(\sqrt{a_{ij}}-\sqrt{a_{ji}}\big)^2
	\quad\mbox{for all }i=1,\ldots,n,
\end{equation}
then the global existence of weak solutions follows.
If $a_{ii}>0$, this condition is satisfied if the matrix $(a_{ij})$ is
nearly symmetric.

The goal of this paper is to prove the global existence of weak solutions
without imposing detailed balance under a condition that is weaker than
\eqref{1.wcd} for (heavily) nonsymmetric matrices $(a_{ij})$.
The key idea of our analysis is the observation that the entropy density
\begin{equation}\label{1.h}
  h(u) = \sum_{i=1}^n \pi_i(u_i-\log u_i)
\end{equation}
formally satisfies the inequality
\begin{equation}\label{1.ei}
  \frac{d}{dt}\int_\Omega h(u(t))dx
	+ \sum_{i=1}^n\pi_i a_{i0}\int_\Omega|\na\log u_i|^2 dx
	+ \sum_{i=1}^n\bigg(8\pi_i a_{ii} - \sum_{j\neq i}\pi_j a_{ji}\bigg)
	\int_\Omega|\na\sqrt{u_i}|^2 dx \le 0.
\end{equation}
(The computation is made rigorous for approximate solutions in \eqref{3.deis}
below.) Thus, we obtain a gradient estimate for $\sqrt{u_i}$ if
\begin{equation}\label{1.aii}
  \kappa := \min_{i=1,\ldots,n}\bigg(8\pi_i a_{ii}
	- \sum_{j=1,\,j\neq i}^n \pi_j a_{ji}\bigg) > 0
	\end{equation}
is satisfied (we allow for $a_{i0}\ge 0$).
If $(a_{ij})$ is almost symmetric, condition \eqref{1.wcd} outperforms
\eqref{1.aii}. However, condition \eqref{1.aii} is generally weaker
than \eqref{1.wcd} if $a_{ij}$ and $a_{ji}$ differ significantly.

We underline this statement by the following example.
Let $n=3$, $a_{13}=a_{21}=a_{32}=1$, and $a_{12}=a_{23}=a_{31}=0$.
Since $|a_{ij}-a_{ji}|=1$ for $i\neq j$, the matrix $(a_{ij})$ is
nonsymmetric.
Condition \eqref{1.wcd} from \cite{CDJ18} is equivalent to $a_{ii}>1/2$ for $i=1,2,3$,
while condition \eqref{1.aii} is equivalent to $a_{11}a_{22}a_{33}>8^{-3}$
(see Lemma \ref{lem.equi} in the Appendix).
This is significantly weaker than $a_{11}a_{22}a_{33}>8^{-1}$ (which follows
from $a_{ii}>1/2$) and, moreover,
we only need {\em one} self-diffusion coefficient to be sufficiently large.

In the literature, the functional \eqref{1.h} has been identified as an entropy
(i.e.\ a Lypunov functional) mainly for higher-order parabolic equations
via the method of systematic integration by parts \cite{JuMa06}.
A similar functional was used to prove the convergence of solutions to the
two-species SKT model to a steady state under quite particular conditions on the
coefficients $a_{ij}$ \cite{JuZa16}.
Up to our knowledge, the use of \eqref{1.h} in the global existence analysis
of cross-diffusion systems is new.

We impose the following assumptions:
\begin{itemize}
\item[(A1)] Domain: $\Omega\subset\R^d$ is a bounded domain with $\pa\Omega\in C^2$,
$d\le 3$, and $T>0$.
\item[(A2)] Initial datum: $u^0=(u_1^0,\ldots,u_n^0)\in L^1(\Omega;\R^n)$
satisfies $u_i>0$ in $\Omega$ for $i=1,\ldots,n$, $\int_\Omega h(u^0)dx<\infty$
if $d\le3$, and moreover $\int_\Omega(u_i^0)^2dx<\infty$ if $d=2,3$.
\item[(A3)] Coefficients: $a_{ij}\ge 0$, $a_{i0}\ge 0$ for all $i,j=1,\ldots,n$,
and there exist $\pi_1,\ldots,\pi_n>0$ such that \eqref{1.aii} holds.
\end{itemize}

The boundary regularity in Assumption (A1) is needed to apply an $H^2(\Omega)$
elliptic regularity result for the duality method
(see the proof of Lemma \ref{lem.est2} below). The restriction to at most three
space dimensions comes from the continuous embedding $H^1(\Omega)\hookrightarrow
L^6(\Omega)$, which is needed to conclude the weak convergence of $(u_i^\delta)^3$
in $L^2(\Omega)$, where $u_i^\delta$ are some approximate solutions;
see Section \ref{sec.approx}, step 2.
In view of the entropy inequality \eqref{1.ei}, we need the regularity
$\int_\Omega h(u^0)dx$ for the initial datum in Assumption (A2).
In two and three space dimensions,
we need more integrability to deal with the quadratic nonlinearity.
As already mentioned, Assumption (A3) is a relaxed ``self-diffusion $>$
cross-diffusion'' condition. Note that the diffusion coefficients $a_{i0}$
are allowed to vanish.

\subsection*{Notation}

A vector-valued function $v:\Omega\to\R^n$ has
the components $v_1,\ldots,v_n$. We denote the entries of a matrix
$A\in\R^{n\times n}$ by $A_{ij}$. We set $Q_T=\Omega\times(0,T)$ for the
space-time cylinder. Furthermore, we need the space of test functions
$$
  W_\nu^{2,p}(\Omega) = \big\{\phi\in W^{2,p}(\Omega):\na\phi\cdot\nu=0
	\mbox{ on }\pa\Omega\big\}, \quad p\ge 2,
$$
and we set $H^2_\nu(\Omega)=W^{2,2}_\nu(\Omega)$.

Our first main result is as follows.

\begin{theorem}[Global existence]\label{thm.ex}
Let Assumptions (A1)--(A3) hold. Then there exists a weak solution $u=(u_1,\ldots,u_n)$
to \eqref{1.eq}--\eqref{1.bic} satisfying $u_i(t)>0$ a.e.\ in $\Omega$,
$\int_\Omega h(u(t))dx<\infty$ for $0<t<T$, the regularity
\begin{align*}
  & u_i\in L^\infty(0,T;L^1(\Omega))\cap L^{3}(Q_T), \quad
	\sqrt{u_i}\in L^2(0,T;H^1(\Omega)), \\
	& \pa_t u_i \in L^{4/3}(0,T;W_\nu^{2,4}(\Omega)'),
\end{align*}
$u$ satisfies the initial conditions in the sense of $W^{2,4}_\nu(\Omega)'$, and
it holds for all $\phi\in L^4(0,T;$ $W^{2,4}_\nu(\Omega;\R^n))$ and
$i=1,\ldots,n$ that
\begin{equation}\label{1.weak}
  \int_0^T\langle \pa_t u_i,\phi_i\rangle dt = \int_0^T\int_\Omega
	u_i p_i(u)\Delta\phi_i dxdt, \quad
	p_i(u) = a_{i0} + \sum_{k=1}^n a_{ik}u_k,
\end{equation}
where $\langle\cdot,\cdot\rangle$ denotes the duality pairing of 
$W^{2,4}_\nu(\Omega)'$ and $W^{2,4}_\nu(\Omega)$.
\end{theorem}

Observe that the weak formulation is weaker than the traditional one.
We can change it, after an integration by parts, to the usual weak formulation
$$
  \int_0^T\langle \pa_t u_i,\phi_i\rangle dt = -\int_0^T\int_\Omega
	\big(p_i(u)\na u_i + u_i\na p_i(u)\big)\cdot\na\phi_i dxdt
$$
for all $\phi_i\in L^\infty(0,T;W^{1,\infty}(\Omega))$,
since $u_i\na u_k=u_i\sqrt{u_k}\na\sqrt{u_k}\in L^1(Q_T)$. The regularity is generally
lower compared to the results in \cite{CDJ18}, where $u_i\in L^2(0,T;H^1(\Omega))$
has been proven. The reason is that the logarithmic entropy \eqref{1.h} can be
interpreted to be of ``zero order'' with respect to $u_i$, while the
Boltzmann entropy $\widetilde{h}(u)$, 
which was used in \cite{CDJ18}, is of ``order one'' in $u_i$.

Theorem \ref{thm.ex} is shown by using the entropy method; see, e.g., \cite{Jue16}.
Since the entropy variable $w_i=(\pa h/\pa u_i)(u)=\pi_i(1-1/u_i)$ is not
invertible for every $w_i\in\R$, we regularize the entropy density by
$h_\eps(u)=h(u)+\eps h^0(u)$ for $\eps>0$, where
$h^0(u)=\sum_{i=1}^n u_i(\log u_i-1)$. Then $h_\eps^{-1}:\R\to(0,\infty)$
is well-defined. However, this generally destroys
the entropy structure in the sense that $A(u)h_\eps''(u)^{-1}$ or, equivalently,
$h_\eps''(u)A(u)$ may be not positive semidefinite.
Therefore, we also regularize $A(u)$ by $A_\eps(u)=A(u)+\eps A^0(u)$,
where $A^0(u)$ is a diagonal matrix with entries $(\mu_i/\pi_i)u_i^2$
and sufficiently large numbers $\mu_i>0$. Lemma \ref{lem.HAeps} below shows that
$h_\eps''(u)A_\eps(u)$ is positive definite, which yields some $L^2(\Omega)$
gradient bounds.
Note that our regularization is simpler than that one used in \cite{CDJ18}.

The estimates from the entropy inequality \eqref{1.ei} are not sufficient to
define $u_ip_i(u)$ from \eqref{1.weak} in $L^1(Q_T)$ in the three-dimensional case,
since the Gagliardo--Nirenberg inequality yields $u_i\in L^{1+2/d}(Q_T)$ only
(see \eqref{3.GN}). To obtain better regularity, we exploit the fact that the
SKT model can be written as $\pa_t u_i = \Delta(u_ip_i(u))$,
which allows us to use the duality method. Basically, we use $(-\Delta)^{-1}u_i$ as
a test function, which leads to an estimate for $u_i^2p_i(u)$ in $L^1(Q_T)$
and, because of $a_{ii}>0$ due to \eqref{1.aii}, 
an estimate for $u_i$ in $L^3(Q_T)$.

\begin{theorem}[Large-time behavior]\label{thm.time}
Let Assumptions (A1)--(A3) hold and suppose that $d=1$ and $a_{i0}>0$
for all $i=1,\ldots,n$.
Let $u$ be the weak solution to \eqref{1.eq}--\eqref{1.bic} constructed in
Theorem \ref{thm.ex} and let $\bar{u}_i=\operatorname{meas}(\Omega)^{-1}
\int_\Omega u_idx$ for $i=1,\ldots,n$. Then
$$
  \lim_{t\to\infty}\|u_i(t)-\bar{u}_i\|_{L^1(\Omega)} = 0.
$$
\end{theorem}

Since $u_i$ conserves the mass, $\bar{u}_i$ is independent of time.
The proof of this result is surprisingly delicate in spite of our restriction
to one space dimension. It is needed to guarantee the continuous embedding
$H^1(\Omega)\hookrightarrow L^\infty(\Omega)$;
see the proof of Lemma \ref{lem.ei.eta}. The usual idea is to show
that the relative entropy, associated to the entropy density \eqref{1.h},
satisfies an inequality similar to \eqref{1.ei} and to estimate the entropy
production term (the gradient bounds) in terms of the relative entropy. Unfortunately,
we have not been able to prove this entropy inequality, since the deregularization
limit in the logarithmic term $\log u_i$ is difficult
and the low integrability of $\pa_t u_i$ and $u_i$ does not allow us to use 
$\pa h(u)/\pa u_i$ as a test function in equation (\ref{1.weak}) to derive 
an entropy inequality. We circumvent this issue by
regularizing the relative entropy:
$$
  \H_\eta(u|\bar{u}) = \sum_{i=1}^n\pi_i
	\int_\Omega(\log(\bar{u}_i+\eta)-\log (u_i+\eta))dx, \quad \eta>0.
$$
The difficult part is to estimate the matrix product
$h_\eps''(u+\eta)A_\eps(u)$. We are able to show that this matrix is positive
definite up to a term of order $O(\sqrt{\eps})$, which vanishes when $\eps\to 0$.
This shows that in the limit $\eps\to 0$, for $0\le s<t$,
$$
  \H_\eta(u(t)|\bar{u})
	+ C\sum_{i=1}^n\int_s^t\int_\Omega|\na\sqrt{u_i+\eta}|^2dxd\sigma
	\le \H_\eta(u(s)|\bar{u}), \quad 0<s<t.
$$
The entropy production can be estimated as (see Lemma \ref{lem.poincare})
$$
	\int_0^\infty\|\sqrt{u_i(t)}-\sqrt{\bar{u}_i}\|_{L^2(\Omega)}^2 dt
	\le C\int_0^\infty\int_\Omega|\na\sqrt{u_i}|^2dxdt \le C(u^0).
$$
Note that the Poincar\'e--Wirtinger inequality would only yield the
difference $\sqrt{u_i(t)}-\overline{\sqrt{u_i}}$.
The previous inequality implies the existence of a sequence
$t_k\to\infty$ as $k\to\infty$ such that
$\|\sqrt{u_i(t_k)}-\sqrt{\bar{u}_i}\|_{L^2(\Omega)}\to 0$.
We will show that this implies the convergence
$\H_\eta(u(t_k)|\bar{u})\to 0$ as $k\to\infty$, and since
the relative entropy is bounded and nonincreasing, this convergence holds for
any sequence $t\to\infty$. Finally, the Csisz\'ar--Kullback inequality
(Proposition \ref{lem.ck} in the Appendix) concludes the proof.

The paper is organized as follows. We prove the positive definiteness of
$h''(u)A(u)$ and $h_\eps''(u)A_\eps(u)$ in Section \ref{sec.aux}. Theorem
\ref{thm.ex} is proved in Section \ref{sec.approx}, while Section \ref{sec.time}
is devoted to the proof of Theorem \ref{thm.time}. Some auxiliary results
are collected in Appendix \ref{sec.app}.


\section{Positive definiteness of mobility matrices}\label{sec.aux}

We introduce the Hessian matrix of $h(u)$, defined in \eqref{1.h},
by $H(u)=h''(u)$ with entries
$H_{ij}(u)=\delta_{ij}\pi_i u_i^{-2}$ for $i,j=1,\ldots,n$.

\begin{lemma}
It holds for any $z\in\R^n$ and $u\in(0,\infty)^n$ that
$$
  z^T H(u)A(u)z \ge \sum_{i=1}^n \pi_i a_{i0}\frac{z_i^2}{u_i^2}
	+ \frac14\sum_{i=1}^n\bigg(8\pi_i a_{ii} - \sum_{j=1,\,j\neq i}^n\pi_j a_{ji}
	\bigg)\frac{z_i^2}{u_i}.
$$
\end{lemma}

\begin{proof}
The elements of the matrix $H(u)A(u)$ equal
\begin{align*}
  (H(u)A(u))_{ij} &= \delta_{ij}\pi_i \frac{a_{i0}}{u_i^{2}}
	+ \delta_{ij}\sum_{k=1}^n \pi_i a_{ik}\frac{u_k}{u_i^{2}}
	+ \pi_i \frac{a_{ij}}{u_i} \\
	&= \delta_{ij}\pi_i  \frac{a_{i0}}{u_i^{2}}
	+ \delta_{ij}\bigg(2\pi_i \frac{a_{ij}}{u_i}
	+ \sum_{k=1,\,k\neq i}^n \pi_i a_{ik}\frac{u_k}{u_i^{2}}\bigg)
	+ (1-\delta_{ij})\pi_i \frac{a_{ij}}{u_i}.
\end{align*}
This gives for all $z\in\R^n$:
\begin{align}\label{2.aux}
  z^T H(u)A(u)z &= \sum_{i=1}^n\pi_i a_{i0}\frac{z_i^2}{u_i^2}
	+ 2\sum_{i=1}^n \pi_i a_{ii}\frac{z_i^2}{u_i} \\
	&\phantom{xx}{}+ \sum_{i=1}^n\sum_{k=1,\,k\neq i}^n\pi_i a_{ik}u_k\frac{z_i^2}{u_i^2}
	+ \sum_{i,j=1,\,i\neq j}^n\pi_i a_{ij} \frac{z_iz_j}{u_i}. \nonumber
\end{align}
We use Young's inequality to estimate the last term:
\begin{align*}
  \sum_{i,j=1,\,i\neq j}^n\pi_i a_{ij} \frac{z_iz_j}{u_i}
	&\ge -\sum_{i,j=1,\,i\neq j}^n\pi_i a_{ij}
	\bigg(\frac{u_j}{u_i^2}z_i^2 + \frac14 \frac{z_j^2}{u_j}\bigg) \\
	&= -\sum_{i,j=1,\,i\neq j}^n\pi_i a_{ij}\frac{u_j}{u_i^2}z_i^2
	- \frac14\sum_{i,j=1,\,i\neq j}^n\pi_j a_{ji}\frac{z_i^2}{u_i}.
\end{align*}
The first term on the right-hand side cancels with the third term on the
right-hand side of \eqref{2.aux}. Therefore,
$$
  z^T H(u)A(u)z \ge \sum_{i=1}^n\pi_i a_{i0}\frac{z_i^2}{u_i^2}
	+ \sum_{i=1}^n\bigg(2\pi_ia_{ii} - \frac14\sum_{i,j=1,\,i\neq j}^n\pi_j a_{ji}\bigg)
	\frac{z_i^2}{u_i}
$$
which finishes the proof.
\end{proof}

For $\eps>0$, we define the approximate entropy density
\begin{equation}\label{2.h0}
  h_\eps(u) = h(u) + \eps h^0(u), \quad\mbox{where }
	h^0(u) = \sum_{i=1}^n u_i(\log u_i-1).
\end{equation}
We set $H^0(u)=(h^0)''(u)$ with entries $H_{ij}^0(u)=\delta_{ij}u_i^{-1}$,
$i,j=1,\ldots,n$, and
$$
  H_\eps(u) = H(u) + \eps H^0(u), \quad\mbox{where }
	H_{\eps,ij}(u) = \delta_{ij}\bigg(\frac{\pi_i}{u_i^2} + \frac{\eps}{u_i}\bigg).
$$
We also need to approximate the diffusion matrix:
$$
  A_\eps(u) = A(u) + \eps A^0(u), \quad\mbox{where }
	A^0_{ij}(u) = \delta_{ij}\frac{\mu_i}{\pi_i}u_i^2,
$$
imposing that $\mu_i\ge \sum_{j\neq i}(a_{ij}+a_{ji})/2$.
The latter condition is necessary to prove that the product $H_\eps(u)A_\eps(u)$
is also positive definite.

\begin{lemma}\label{lem.HAeps}
It holds for any $z\in\R^n$ and $u\in(0,\infty)^n$ that
$$
  z^T H_\eps(u)A_\eps(u)z \ge z^TH(u)A(u)z + 2\eps\sum_{i=1}^n a_{ii}z_i^2
	+ \eps^2\sum_{i=1}^n\frac{\mu_i}{\pi_i}u_iz_i^2.
$$
\end{lemma}

\begin{proof}
We decompose the product $H_\eps(u)A_\eps(u)$ as
$$
  H_\eps(u)A_\eps(u) = H(u)A(u) + \eps\big(H^0(u)A(u) + H(u)A^0(u)\big)
	+ \eps^2 H^0(u)A^0(u).
$$
We compute first the terms of order $\eps$:
\begin{align*}
  (H^0(u)A(u))_{ij} &= \delta_{ij}\bigg(\frac{a_{i0}}{u_i}
	+ \sum_{k\neq i}a_{ik}\frac{u_k}{u_i} + 2a_{ii}\bigg) + (1-\delta_{ij})a_{ij}, \\
	(H(u)A^0(u))_{ij} &= \delta_{ij}\mu_i,
\end{align*}
which yields
\begin{align*}
  z^T\big(H^0(u)A(u) + H(u)A^0(u)\big)z
	&= \sum_{i=1}^n\bigg(\frac{a_{i0}}{u_i} + \sum_{k\neq i}a_{ik}\frac{u_k}{u_i}
	+ 2a_{ii} + \mu_i\bigg)z_i^2 \\
	&\phantom{xx}{}+ \sum_{i=1}^n\sum_{j=1,\,j\neq i}^n a_{ij}z_iz_j \\
	&\ge \sum_{i=1}^n\bigg(2a_{ii} + \mu_i\bigg)z_i^2
	+ \sum_{i=1}^n\sum_{j=1,\,j\neq i}^n a_{ij}z_iz_j.
\end{align*}
The last term is estimated by using Young's inequality
$z_iz_j\ge -(z_i^2+z_j^2)/2$:
\begin{align*}
  \sum_{i=1}^n\sum_{j=1,\,j\neq i}^n a_{ij}z_iz_j
	&\ge -\frac12\sum_{i=1}^n\sum_{j=1,\,j\neq i}^n(a_{ij}z_i^2 + a_{ij}z_j^2) \\
	&= -\frac12\sum_{i=1}^n\sum_{j=1,\,j\neq i}^n a_{ij}z_i^2
	- \frac12\sum_{j=1}^n\sum_{i=1,\,i\neq j}^n a_{ji}z_i^2,
\end{align*}
which, because of the choice of $\mu_i$, shows that
\begin{align*}
  z^T\big(H^0(u)A(u) + H(u)A^0(u)\big)z
	&\ge 2\sum_{i=1}^n a_{ii}z_i^2 + \sum_{i=1}^n\bigg(\mu_i
	- \frac12\sum_{j=1,\,j\neq i}^n(a_{ij}+a_{ji})\bigg)z_i^2 \\
	&\ge 2\sum_{i=1}^n a_{ii}z_i^2.
\end{align*}
The $\eps^2$-term becomes $z^T H^0(u)A^0(u)z=\sum_{i=1}^n(\mu_i/\pi_i)u_iz_i^2$.
Collecting these terms, the proof follows.
\end{proof}


\section{Proof of Theorem \ref{thm.ex}}\label{sec.approx}

Let $T>0$, $N\in\N$, $\tau=T/N>0$, $\delta>0$, and $\eps>0$.
Let $u_\eps^0=(u_{\eps,1}^0,\ldots,u_{\eps,n}^0)$ be a componentwise bounded
sequence of functions with positive lower bounds satisfying 
$h(u_\eps^0)\to h(u^0)$ strongly in
$L^1(\Omega)$ and $u_\eps^0\to u^0$ strongly in $L^2(\Omega)$ as $\eps\to 0$.

{\em Step 1: Solution of an approximated problem.}
We introduce the entropy variables $w_i=(\pa h_\eps/\pa u_i)(u)
=\pi_i(1-1/u_i)+\eps\log u_i$, $i=1,\ldots,n$. Since the range of $h'_\eps$ is
$\R^n$, the transformation $u:\R^n\to(0,\infty)^n$, $u(w)=(h'_\eps)^{-1}(w)$,
is well defined. Furthermore, we introduce the mobility matrix
$B_\eps(w)=A_\eps(u(w))H_\eps(u(w))^{-1}$. By construction of $u_\eps^0$,
we can define $w^0=h'_\eps(u_\eps^0)$, and this is an element of
$L^\infty(\Omega;\R^n)$. Then $u(w^0)=u_\eps^0$.
Let $m=1$ if $d=1$ and $m=2$ if $d=2,3$.
Given $k\in\N$ and $w^{k-1}\in L^\infty(\Omega;\R^n)$,
we wish to find $w^k\in H^m(\Omega;\R^n)$ solving
\begin{align}\label{3.eps}
  \frac{1}{\tau}\int_\Omega&(u(w^k)-u(w^{k-1}))\cdot\phi dx
	+ \int_\Omega\na\phi:B_\eps(w^k)\na w^k dx \\
	&{}+ \delta\int_\Omega\bigg(\sum_{|\alpha|=m}D^\alpha w^k\cdot D^\alpha\phi
	+ w^k\cdot\phi\bigg)dx = 0 \nonumber
\end{align}
for all $\phi\in H^m(\Omega;\R^n)$, where $\alpha=(\alpha_1,\ldots,\alpha_d)\in\N_0^d$
is a multiindex and
$D^\alpha$ equals the partial derivative
$\pa^{|\alpha|}/\pa x_1^{\alpha_1}\cdots\pa x_d^{\alpha_d}$.

We claim that the existence of a weak solution $w^k$ follows from
\cite[Lemma 5]{Jue15}. The construction of $h_\eps$ ensures that Hypothesis
H1 of \cite{Jue15} is satisfied. Lemma \ref{lem.HAeps} shows that
Hypothesis H2 holds as well. Also Hypothesis H3 is fulfilled since \eqref{1.eq}
does not contain any source terms. We deduce from \cite[Lemma 5]{Jue15} that
there exists a weak solution $w^k\in H^m(\Omega;\R^n)$ to \eqref{3.eps}, satisfying
the discrete entropy inequality
\begin{align}\label{3.dei}
  \int_\Omega &h_\eps(u(w^k))dx + \tau\int_\Omega\na w^k:B_\eps(w^k)\na w^k dx \\
	&{}+ \delta\tau\int_\Omega\bigg(\sum_{|\alpha|=m}|D^\alpha w^k|^2 + |w^k|^2\bigg)dx
	\le \int_\Omega h_\eps(u(w^{k-1}))dx. \nonumber
\end{align}

We derive some estimates for $w^k$ and $u^k:=u(w^k)$.
According to Lemma \ref{lem.HAeps}, the second term in \eqref{3.dei} can be estimated
as follows:
\begin{align}\label{3.nablaw}
  \int_\Omega&\na w^k:B_\eps(w^k)\na w^k dx
	= \int_\Omega \na u^k:H_\eps(u^k)A_\eps(u^k)\na u^k dx \\
	&\ge \sum_{i=1}^n\int_\Omega\Big\{
	\kappa|\na(u_i^k)^{1/2}|^2+ 2\eps\Big(\min_{i=1,\ldots,n}a_{ii}\Big)|\na u_i^k|^2
	\Big\}dx, \nonumber
\end{align}
recalling definition \eqref{1.aii} of $\kappa$.
Therefore, since $a_{ii}>0$ by Assumption (A3), summing \eqref{3.dei} over
$k=1,\ldots,j$,
\begin{align}\label{3.deis}
  \int_\Omega & h_\eps(u^j)dx + C\tau\sum_{k=1}^j\sum_{i=1}^n\int_\Omega
	\big(|\na(u_i^k)^{1/2}|^2 + \eps|\na u_i^k|^2\big)dx \\
	&{}+ \delta\tau\sum_{k=1}^j\int_\Omega\bigg(\sum_{|\alpha|=m}|D^\alpha w^k|^2
	+ |w^k|^2\bigg)dx \le \int_\Omega h_\eps(u^0_\eps)dx \le C, \nonumber
\end{align}
where $C>0$ denotes here and in the following a constant which is
independent of $\delta$, $\eps$, and $\tau$ with values changing from line to line.

To derive bounds in $H^1(\Omega)$, we apply the Poincar\'e--Wirtinger inequality
for which we need a uniform estimate for $u_i^j$. We take the test function
$\phi=(\delta_{i1},\ldots,\delta_{in})$ in \eqref{3.eps} and sum the
resulting equation over $k=1,\ldots,j$. Then, taking into account \eqref{3.deis},
\begin{align}\label{3.L1}
  0 &\le \int_\Omega u_i^j dx = \int_\Omega u_i^0 dx
	- \delta\tau\sum_{k=1}^j\int_\Omega w_i^k dx \\
	&\le \int_\Omega u_i^0 dx + \frac{\delta}{2}\tau\sum_{k=1}^j\int_\Omega
	((w_i^k)^2 + 1)dx \le C(u^0,T,\Omega). \nonumber
\end{align}
We infer that
\begin{equation}\label{3.uw}
  \|u^k\|_{H^1(\Omega)}\le C(\eps,\tau), \quad
	\sqrt{\delta}\|w^k\|_{H^m(\Omega)} \le C(\tau).
\end{equation}

{\em Step 2: Limit $\delta\to 0$.} Let $w^\delta:=w^k$ and $u^\delta:=u^k$.
Before we pass to the limit $\delta\to 0$, we derive
a very weak formulation for $u^\delta$. It holds that
\begin{align*}
  (B_\eps(w^\delta)\na w^{\delta})_i
	&= (A_\eps(u^{\delta})\na u^{\delta})_i
	= \eps(A^0(u^{\delta})\na u^{\delta})_i + \na(u_i^{\delta}p_i(u^{\delta})) \\
	&= \frac{\eps}{3}\frac{\mu_i}{\pi_i}\na(u_i^{\delta})^3
	+ \na(u_i^{\delta}p_i(u^{\delta})).
\end{align*}
Therefore, in view of \eqref{3.eps}, $(u^\delta,w^\delta)$ solves for all
$\phi\in H^2_\nu(\Omega;\R^n)$,
\begin{align}\label{3.delta}
  \frac{1}{\tau}\int_\Omega&(u^\delta-u^{k-1})\cdot\phi dx
	-\sum_{i=1}^n\int_\Omega\bigg(\frac{\eps}{3}\frac{\mu_i}{\pi_i}(u_i^\delta)^3
	+ u_i^\delta p_i(u^\delta)\bigg)\Delta\phi_i dx \\
	&{}+\delta\int_\Omega\bigg(\sum_{|\alpha|=m}D^\alpha w^\delta\cdot D^\alpha\phi
	+ w^\delta\cdot\phi\bigg)dx = 0. \nonumber
\end{align}
In view of the uniform bounds \eqref{3.uw} and using the compact embedding
$H^1(\Omega)\hookrightarrow L^4(\Omega)$ (if $d\le 3$),
there exist subsequences of $(u^\delta)$ and $(w^\delta)$, which are not relabeled,
such that, as $\delta\to 0$,
$$
  u^\delta\to u \quad\mbox{strongly in }L^4(\Omega), \quad
	\delta w^\delta\to 0\quad\mbox{strongly in }H^m(\Omega).
$$
It follows from the linearity of $p_i$ that
$u_i^\delta p_i(u^\delta)\to u_ip_i(u)$ strongly
in $L^2(\Omega)$. Moreover, up to a subsequence, $u^\delta\to u$ a.e.\ in
$\Omega$ and, because of the continuous embedding $H^1(\Omega)\hookrightarrow
L^6(\Omega)$ for $d\le 3$,
$(u_i^\delta)^3\rightharpoonup u_i^3$ weakly in $L^2(\Omega)$.
Thus, passing to the limit $\delta\to 0$ in \eqref{3.delta},
we find that, for all $\phi\in H^2_\nu(\Omega;\R^n)$,
\begin{equation}\label{3.uk}
  \frac{1}{\tau}\int_\Omega(u^k-u^{k-1})\cdot\phi dx
	= \sum_{i=1}^n\int_\Omega\bigg(\frac{\eps}{3}\frac{\mu_i}{\pi_i}(u_i^k)^3
	+ u_i^k p_i(u^k)\bigg)\Delta\phi_i dx,
\end{equation}
where we have set $u^k:=u$.


{\em Step 3: Bounds uniform in $(\eps,\tau)$.}
We introduce piecewise in time constant functions and formulate some
bounds uniform in $(\eps,\tau)$. Let
$u^{(\tau)}(x,t)=u(x)$ for $x\in\Omega$,
$t\in((k-1)\tau,k\tau]$. At time $t=0$, we set
$u^{(\tau)}(\cdot,0)=u^0_\eps$. Furthermore, let
$u^{(\tau)}=(u_1^{(\tau)},\ldots,u_n^{(\tau)})$.
We define the backward shift operator
$(\sigma_\tau u^{(\tau)})(x,t) = u^{k-1}(x)$ for $x\in\Omega$,
$t\in((k-1)\tau,k\tau]$. In view of \eqref{3.uk}, $u^{(\tau)}$ solves
\begin{align}\label{3.aux2}
  \frac{1}{\tau}&\int_0^T\int_\Omega(u^{(\tau)}-\sigma_\tau u^{(\tau)})\cdot\phi dxdt \\
	&= \sum_{i=1}^n\int_0^T\int_\Omega
	\bigg(\frac{\eps}{3}\frac{\mu_i}{\pi_i}(u_i^{(\tau)})^3
	+ u_i^{(\tau)} p_i(u^{(\tau)})\bigg)\Delta \phi_i dxdt \nonumber
\end{align}
for piecewise constant functions $\phi:(0,T)\to H^2_\nu(\Omega;\R^n)$.
By a density argument, this equation also holds for all
$\phi\in L^2(0,T;H^2_\nu(\Omega;\R^n))$ \cite[Prop.~1.36]{Rou05}.

We conclude from the summarized discrete entropy inequality \eqref{3.deis},
the $L^1(\Omega)$ estimate \eqref{3.L1}, and the Poincar\'e--Wirtinger
inequality the following $(\eps,\tau)$-independent bounds.

\begin{lemma}\label{lem.est1}
There exists a constant $C>0$, which is independent of $\eps$ and $\tau$, such that
for all $i=1,\ldots,n$,
$$
  \|u_i^{(\tau)}\|_{L^\infty(0,T;L^1(\Omega))}
	+ \|(u_i^{(\tau)})^{1/2}\|_{L^2(0,T;H^1(\Omega))}
	+ \sqrt{\eps}\|u_i^{(\tau)}\|_{L^2(0,T;H^1(\Omega))} \le C.
$$
\end{lemma}

The Gagliardo--Nirenberg inequality for $p=1+2/d$ and $\theta=d/(2+d)$ gives
\begin{align}\label{3.GN}
  \|u_i^{(\tau)}\|_{L^p(Q_T)}^p
	&= \int_0^T\|(u_i^{(\tau)})^{1/2}\|_{L^{2p}(\Omega)}^{2p} dt
	\le C\int_0^T\|(u_i^{(\tau)})^{1/2}\|_{H^1(\Omega)}^{2p\theta}
	\|(u_i^{(\tau)})^{1/2}\|_{L^2(\Omega)}^{2p(1-\theta)}dt \\
	&\le C \|(u_i^{(\tau)}\|_{L^\infty(0,T;L^1(\Omega)}^{p(1-\theta)}
	\int_0^T\|(u_i^{(\tau)})^{1/2}\|_{H^1(\Omega)}^{2}dt \le C, \nonumber
\end{align}
since $2p\theta=2$. As we need at least a uniform estimate for $u_i^{(\tau)}$
in $L^{2+\eta}(Q_T)$ for $\eta>0$
to pass to the limit in \eqref{3.aux2}, the above $L^p(Q_T)$ bound
is not sufficient except for $d=1$. We need an additional estimate,
which is provided by the following lemma.

\begin{lemma}\label{lem.est2}
There exists a constant $C>0$, which is independent of $\eps$ and $\tau$, such that
$$
  \|u_i^{(\tau)}\|_{L^3(Q_T)} + \eps^{1/4}\|u_i^{(\tau)}\|_{L^4(Q_T)} \le C,
	\quad i=1,\ldots,n.
$$
\end{lemma}

\begin{proof}
We use the duality method. For this, let $\psi_i^k\in\{\psi\in H^2_\nu(\Omega):
\int_\Omega\psi dx=0\}$ be the unique solution to
\begin{equation}\label{3.psi}
  -\Delta\psi_i^k = u_i^k - \fint_\Omega u_i^k dx\quad\mbox{in }\Omega, \quad
	\na\psi_i^k\cdot\nu = 0 \quad\mbox{on }\pa\Omega,
\end{equation}
where $\fint u_i^k dx = \operatorname{meas}(\Omega)^{-1}\int_\Omega u_i^k dx$.
This problem is well-posed since $u_i^k\in L^2(\Omega)$ and $\pa\Omega\in C^2$.
We use $\psi_i^k$ as a test function in the weak formulation \eqref{3.psi}:
\begin{align*}
  \int_\Omega|\na\psi_i^k|^2 dx &= \int_\Omega u_i^k\psi_i^k dx
	- \fint_\Omega u_i^k dx\int_\Omega\psi_i^k dx
	= \int_\Omega u_i^k\psi_i^k dx \\
	&\le \|u_i^k\|_{L^2(\Omega)}\|\psi_i^k\|_{L^2(\Omega)}
	\le C\|u_i^k\|_{L^2(\Omega)}\|\na\psi_i^k\|_{L^2(\Omega)},
\end{align*}
where we applied the Poincar\'e--Wirtinger inequality in the last step. Thus,
$\|\na\psi_i^k\|_{L^2(\Omega)}\le C\|u_i^k\|_{L^2(\Omega)}$ and,
by the Poincar\'e inequality again, $\|\psi_i^k\|_{L^2(\Omega)}\le
C\|u_i^k\|_{L^2(\Omega)}$. Hence,
$$
  \|\psi_i^k\|_{H^1(\Omega)}\le C\|u_i^k\|_{L^2(\Omega)}.
$$

Now, taking $\phi_i=\psi_i^k$ and $\phi_j=0$ for $j\neq i$
as a test function in the weak formulation of \eqref{3.uk} and using
equation \eqref{3.psi} for $u_i^k$ and the property $\int_\Omega\psi_i^k dx=0$,
\begin{align*}
  -\frac{1}{\tau}&\int_\Omega(\Delta\psi_i^k-\Delta\psi_i^{k-1})\psi_i^k dx
	= \frac{1}{\tau}\int_\Omega(u_i^k-u_i^{k-1})\psi_i^k dx
	- \frac{1}{\tau}\fint_\Omega(u_i^k-u_i^{k-1})dx\int_\Omega\psi_i^k dx \\
	&= \int_\Omega\bigg(\frac{\eps}{3}\frac{\mu_i}{\pi_i}(u_i^k)^3
	+ u_i^k p_i(u^k)\bigg)\Delta\psi_i^k dx \\
	&= -\int_\Omega\bigg(\frac{\eps}{3}\frac{\mu_i}{\pi_i}(u_i^k)^3
	+ u_i^k p_i(u^k)\bigg)u_i^k dx
	+ \fint_\Omega u_i^k dx\int_\Omega\bigg(\frac{\eps}{3}\frac{\mu_i}{\pi_i}(u_i^k)^3
	+ u_i^k p_i(u^k)\bigg)dx.
\end{align*}
Summing this identity over $k=1,\ldots,N$ and observing that
\begin{align*}
  -\frac{1}{\tau}\int_\Omega(\Delta\psi_i^k-\Delta\psi_i^{k-1})\psi_i^k dx
	&= \frac{1}{\tau}\int_\Omega(|\na\psi_i^k|^2 - \na\psi_i^{k-1}\cdot\na\psi_i^k)dx \\
	&\ge \frac{1}{2\tau}\int_\Omega(|\na\psi_i^k|^2 - |\na\psi_i^{k-1}|^2)dx,
\end{align*}
we obtain
\begin{align}\label{3.aux3}
  \frac{1}{2}\int_\Omega(&|\na\psi_i^N|^2 - |\na\psi_i^0|^2)dx
	\le -\int_0^T\int_\Omega\bigg(\frac{\eps}{3}\frac{\mu_i}{\pi_i}(u_i^{(\tau)})^4
	+ (u_i^{(\tau)})^2 p_i(u^{(\tau)})\bigg)dxdt \\
	&{}+ \int_0^T\bigg(\fint_\Omega u_i^{(\tau)}dx\bigg)\int_\Omega
	\bigg(\frac{\eps}{3}\frac{\mu_i}{\pi_i}(u_i^{(\tau)})^3
	+ u_i^{(\tau)} p_i(u^{(\tau)})\bigg)dxdt. \nonumber
\end{align}
As $u_i^{(\tau)}$ is bounded in $L^\infty(0,T;L^1(\Omega))$ by Lemma \ref{lem.est1},
we can estimate the last term on the right-hand side by
\begin{align*}
  \int_0^T\bigg(&\fint_\Omega u_i^{(\tau)}dx\bigg)\int_\Omega
	\bigg(\frac{\eps}{3}\frac{\mu_i}{\pi_i}(u_i^{(\tau)})^3
	+ u_i^{(\tau)} p_i(u^{(\tau)})\bigg)dxdt \\
	&\le C(u^0)\int_0^T\int_\Omega\bigg(\frac{\eps}{3}\frac{\mu_i}{\pi_i}(u_i^{(\tau)})^3
	+ u_i^{(\tau)} p_i(u^{(\tau)})\bigg)dxdt.
\end{align*}
We deduce from Young's inequality $ab\le (\delta a)^p/p + (b/\delta)^q/q$ for
$a$, $b\ge 0$ and $1/p+1/q=1$ for suitable $\delta>0$ that
\begin{align*}
  C(u^0)\frac{\eps}{3}\frac{\mu_i}{\pi_i}(u_i^{(\tau)})^3
	&\le \frac{\eps}{6}\frac{\mu_i}{\pi_i}(u_i^{(\tau)})^4 + C_1, \\
	C(u^0) u_i^{(\tau)}p_i(u^{(\tau)}) &\le \frac12(u_i^{(\tau)})^2p_i(u^{(\tau)})
	+ C_2 p_i(u^{(\tau)}).
\end{align*}
The first terms on the right-hand sides can be absorbed by the first term on the
right-hand side of \eqref{3.aux3}, leading to
\begin{align*}
  \frac{1}{2}\int_\Omega(|\na\psi_i^N|^2 - |\na\psi_i^0|^2)dx
	&\le -\int_0^T\int_\Omega\bigg(\frac{\eps}{6}\frac{\mu_i}{\pi_i}(u_i^{(\tau)})^4
	+ \frac12(u_i^{(\tau)})^2 p_i(u^{(\tau)})\bigg)dx \\
	&\phantom{xx}{}+ \int_0^T\int_\Omega(C_1+C_2 p_i(u^{(\tau)}))dx.
\end{align*}
Since $p_i(u^{(\tau)})$ depends linearly on $u_i^{(\tau)}$ and this function
is uniformly bounded in $L^\infty(0,T;$ $L^1(\Omega))$, we conclude that
\begin{align*}
  \frac{1}{2}\int_\Omega&|\na\psi_i^N|^2 dx
	+ \int_0^T\int_\Omega\bigg(\frac{\eps}{6}\frac{\mu_i}{\pi_i}(u_i^{(\tau)})^4
	+ \frac12(u_i^{(\tau)})^2 p_i(u^{(\tau)})\bigg)dx \\
	&\le \frac{1}{2}\int_\Omega|\na\psi_i^0|^2dx + C_3(u^0)
	\le C\|u_{\eps,i}^0\|_{L^2(\Omega)}^2 + C_3(u^0) \le C(u^0).
\end{align*}
Taking into account the inequality
$(u_i^{(\tau)})^2p_i(u^{(\tau)})\ge a_{ii}(u_i^{(\tau)})^3$,
this finishes the proof.
\end{proof}

\begin{lemma}\label{lem.est3}
There exists a constant $C>0$, which is independent of $\eps$ and $\tau$, such that
$$
  \|u_i^{(\tau)}\|_{L^{3/2}(0,T;W^{1,3/2}(\Omega))} \le C.
$$
\end{lemma}

\begin{proof}
This estimate follows directly from Lemmas \ref{lem.est1}--\ref{lem.est2}
and H\"older's inequality:
\begin{align*}
  \|\na u_i^{(\tau)}\|_{L^{3/2}(Q_T)}
	&= 2\|(u_i^{(\tau)})^{1/2}\na(u_i^{(\tau)})^{1/2}\|_{L^{3/2}(Q_T)}, \\
	&\le 2\|u_i^{(\tau)}\|_{L^3(Q_T)}^{1/2}\|\na(u_i^{(\tau)})^{1/2}\|_{L^2(Q_T)} \le C,
\end{align*}
as well as the bound for $u_i^{(\tau)}$ in $L^3(Q_T)$ and also in $L^{3/2}(Q_T)$.
\end{proof}

The improved integrability of $u^{(\tau)}$ provides a uniform bound for
the discrete time derivative.

\begin{lemma}\label{lem.est4}
There exists a constant $C>0$, which is independent of $\eps$ and $\tau$, such that
$$
  \tau^{-1}\|u^{(\tau)}-\sigma_\tau u^{(\tau)}\|_{L^{4/3}(0,T;W_\nu^{2,4}(\Omega)')}
	\le C.
$$
\end{lemma}

\begin{proof}
Lemma \ref{lem.est2} shows that
$$
  \|u_i^{(\tau)}p_i(u^{(\tau)})\|_{L^{3/2}(Q_T)} \le C, \quad
  \eps\|(u_i^{(\tau)})^3\|_{L^{4/3}(Q_T)} = \eps\|u_i^{(\tau)}\|_{L^4(Q_T)}^3
	\le \eps^{1/4}C.
$$
Let $\phi\in L^4(0,T;W_\nu^{2,4}(\Omega;\R^n))$. Then, using \eqref{3.aux2}, we
can estimate as follows:
\begin{align*}
  \frac{1}{\tau}&\bigg|\int_0^T\int_\Omega(u^{(\tau)}-\sigma_\tau u^{(\tau)})
	\cdot\phi dxdt\bigg| \le \sum_{i=1}^n\|u_i^{(\tau)}p_i(u^{(\tau)})\|_{L^{3/2}(Q_T)}
	\|\Delta\phi_i\|_{L^3(Q_T)} \\
	&\phantom{xx}{}+\frac{\eps}{3}\sum_{i=1}^n\frac{\mu_i}{\pi_i}
	\|(u_i^{(\tau)})^3\|_{L^{4/3}(Q_T)}\|\Delta\phi_i\|_{L^4(Q_T)}
	\le C(1+\eps^{1/4})\|\phi\|_{L^4(0,T;W^{2,4}(\Omega))}.
\end{align*}
This finishes the proof.
\end{proof}

{\em Step 4: Limit $(\eps,\tau)\to 0$.}
Lemmas \ref{lem.est3} and \ref{lem.est4} allow us to apply the lemma of
Aubin--Lions in the version of \cite{DrJu12}, yielding the existence of a
subsequence of $(u^{(\tau)})$, which is not relabeled, such that,
as $(\eps,\tau)\to 0$,
$$
  u^{(\tau)}\to u\quad\mbox{strongly in }L^{3/2}(Q_T)\mbox{ and a.e.\ in }Q_T.
$$
It follows from Lemmas \ref{lem.est1} and \ref{lem.est4} that
\begin{align*}
  (u_i^{(\tau)})^{1/2}\rightharpoonup u_i^{1/2} &\quad\mbox{weakly in }
	L^2(0,T;H^1(\Omega)), \\
  \tau^{-1}(u^{(\tau)}-\sigma_\tau u^{(\tau)})
	\rightharpoonup \pa_t u &\quad\mbox{weakly in }
	L^{4/3}(0,T;W^{2,4}_\nu(\Omega)').
\end{align*}
The a.e.\ convergence of $(u^{(\tau)})$ implies that
$u_i^{(\tau)}p_i(u^{(\tau)})\to u_ip_i(u)$ a.e.\ in $Q_T$.
Since $u_i^{(\tau)}p_i(u^{(\tau)})$ is bounded in $L^{3/2}(Q_T)$, we infer that
$$
  u_i^{(\tau)}p_i(u^{(\tau)})\to u_ip_i(u)\quad\mbox{strongly in }L^{4/3}(Q_T).
$$
Furthermore, taking into account Lemma \ref{lem.est2}, as $\eps\to 0$,
$$
  \eps\|(u_i^{(\tau)})^3\|_{L^{4/3}(Q_T)}
	= \eps^{1/4}\big(\eps^{1/4}\|u_i^{(\tau)}\|_{L^4(Q_T)}\big)^3 \le C\eps^{1/4}
	\to 0.
$$
Thus, performing the limit $(\eps,\tau)\to 0$ in \eqref{3.aux2} shows that
$u$ solves \eqref{1.weak}. As $u_i\in W^{1,4/3}(0,T;$ $W^{2,4}_\nu(\Omega)')
\hookrightarrow C^0([0,T];W^{2,4}_\nu(\Omega)')$, the initial condition is
satisfied in the sense of $W_\nu^{2,4}(\Omega)'$.

\begin{remark}[One-dimensional case]\label{rem}\rm
The additional regularity from the duality method is not needed in the
one-dimensional case. In that case, the proof simplifies. First, we may
choose $\delta=\eps$. Second, the Gagliardo--Nirenberg inequality \eqref{3.GN}
shows that $u_i^{(\tau)}$ is uniformly bounded in $L^3(Q_T)$. Furthermore, by estimate
\eqref{3.nablaw}, $\sqrt{\eps}u_i^{(\tau)}$ is uniformly bounded in
$L^2(0,T;H^1(\Omega))$. Hence, using the Gagliardo--Nirenberg inequality
with $\theta=1/2$ and the uniform bounds in Lemma \ref{lem.est1},
\begin{align*}
  \eps \|u_i^{(\tau)}\|_{L^4(Q_T)}^4
	&\le \eps C\int_0^T\|u_i^{(\tau)}\|_{H^1(\Omega)}^{4\theta}
	\|u_i^{(\tau)}\|_{L^1(\Omega)}^{4(1-\theta)} dt \\
	&\le \eps C\|u_i^{(\tau)}\|_{L^\infty(0,T;L^1(\Omega))}^2
	\int_0^T\|u_i^{(\tau)}\|_{H^1(\Omega)}^2 dt \le C.
\end{align*}
This shows that $\eps^{1/4}u_i^{(\tau)}$ is uniformly bounded in $L^4(Q_T)$,
and we obtain the same estimates as in Lemma \ref{lem.est2}, which allow us
to conclude.
\qed
\end{remark}


\section{Large-time behavior}\label{sec.time}

In this section, we prove Theorem \ref{thm.time}.
First, we show an entropy inequality which gives time-uniform estimates.

\begin{lemma}[Entropy inequality I]\label{lem.ei}
It holds for all $t>0$ that
$$
  \int_\Omega h(u(t))dx + C\sum_{i=1}^n\int_0^t\int_\Omega|\na\sqrt{u_i}|^2 dx
	\le \int_\Omega h(u^0)dx.
$$
\end{lemma}

\begin{proof}
We find from \eqref{3.deis} that
$$
  \int_\Omega h_\eps(u^{(\tau)}(t))dx + C\sum_{i=1}^n\int_0^t\int_\Omega
	|\na(u_i^{(\tau)})^{1/2}|^2dx \le \int_\Omega h_\eps(u^{(\tau)}(0))dx,
$$
where $t\in((j-1)\tau,j\tau]$.
Recalling that $h_\eps(u)=h(u)+\eps h^0(u)$ and $h^0(u)\ge -n$ (see \eqref{2.h0}),
it follows that
\begin{equation}\label{4.aux}
  \int_\Omega h(u^{(\tau)}(t))dx + C\sum_{i=1}^n\int_0^t\int_\Omega
	|\na(u_i^{(\tau)})^{1/2}|^2dx \le \int_\Omega h_\eps(u^0_\eps)dx + C\eps.
\end{equation}
Because of the a.e.\ convergence of $(u^{(\tau)})$, we have
$h(u^{(\tau)}(t))\to h(u(t))$ in $\Omega$ for a.e.\ $t\in(0,T)$, such that
Fatou's lemma implies that
$$
  \int_\Omega h(u(t))dx \le \liminf_{\tau\to 0}\int_\Omega h(u^{(\tau)}(t))dx
$$
for a.e.\ $t>0$. Then, using the weak lower semicontinuity of the $L^2(Q_T)$ norm,
we infer from \eqref{4.aux} in the limit $(\eps,\tau)\to 0$ the conclusion.
\end{proof}

The following lemma is a consequence of Lemma \ref{lem.ei}. Both Lemma \ref{lem.ei}
and \ref{lem.log} are valid in several space dimensions.

\begin{lemma}\label{lem.log}
There exists a constant $C>0$, only depending on $u^0$, such that
$$
  \|u_i\|_{L^\infty(0,\infty;L^1(\Omega))}
	+ \|\log u_i\|_{L^\infty(0,\infty;L^1(\Omega))} 
	+ \|\nabla \sqrt{u_i}\|_{L^2(0,\infty;L^2(\Omega))}\le C.
$$
\end{lemma}

\begin{proof}
The elementary inequalities $z-\log z\ge|\log z|$ and $z-\log z\ge z/2$ for $z>0$,
together with Lemma \ref{lem.ei}, imply that
$$
  \|\log u_i\|_{L^\infty(0,\infty;L^1(\Omega))} 
	\le \frac{1}{\pi_i}\int_\Omega h(u^0)dx, \quad
	\|u_i\|_{L^\infty(0,\infty;L^1(\Omega))} \le \frac{2}{\pi_i}\int_\Omega h(u^0)dx.
$$
The bound on $\na\sqrt{u_i}$ is a consequence of Lemma \ref{lem.ei}.
\end{proof}

It is essential to use the entropy at time $s=0$ in Lemma \ref{lem.ei}
because it is unclear how to pass to limit $(\tau,\eps)\to 0$ in
the entropy at time $s>0$, as $u_i(s)$ may vanish on a set of zero measure.
We overcome this issue by using the test function 
$$
  \bigg(\pi_i\bigg(1-\frac{1}{u_i^k+\eta}\bigg)+\eps\log(u_i^k+\eta)
	\bigg)_{i=1,\ldots,n}
$$
for $\eta>0$ in \eqref{3.eps}. This means that we need to estimate the matrix
$H_\eps(u+\eta)A_\eps(u)$, where $u+\eta=(u_1+\eta,\ldots,u_n+\eta)$,
similarly as we estimated $H_\eps(u)A_\eps(u)$ in Lemma \ref{lem.HAeps}.
This is done in the following lemma.

\begin{lemma}\label{lem.HAetaeps}
There exists $\eta_0>0$ such that for all $0<\eta\le\eta_0$,
$u\in(0,\infty)^n$, and $z\in\R^n$,
it holds that
$$
  z^T H_\eps(u+\eta)A_\eps(u)z \ge \frac{\kappa}{4}\sum_{i=1}^n\frac{z_i^2}{u_i+\eta}
	- \eta\eps C_1\sum_{i=1}^n\frac{z_i^2}{u_i+\eta}
	- \eta\eps^2 C_2\sum_{i=1}^n z_i^2,
$$
where $C_1>0$ depends on $(a_{ij})$, $(\mu_i)$ and $C_2>0$ depends on $(\mu_i/\pi_i)$.
\end{lemma}

\begin{proof}
We decompose the matrix $A_\eps(u)=A(u)+\eps A^0(u)$ as follows:
\begin{align*}
  & A_\eps(u) = A_{\eps,\eta}(u) - \eta A^1, \quad\mbox{where }
	A_{\eps,\eta}(u) := A(u) + \eta A^1 + \eps A^0(u+\eta) - \eps A^2(u), \\
  & A^1_{ij} := \delta_{ij}\bigg(\sum_{k=1}^n a_{ik} + a_{ij}\bigg), \quad
	A^2_{ij}(u) := \delta_{ij}\eta\mu_i\pi_i^{-1}(2u_i+\eta).
\end{align*}
Note that we have written the matrix $A^0(u)$ as $A^0(u) = A^0(u+\eta)-A^2(u)$
and that we have added and subtracted the matrix $\eta A^1$.
We wish to estimate
\begin{align}\label{4.HepsAeps}
  H_\eps(u+\eta)A_\eps(u) &= \big(H(u+\eta)+\eps H^0(u+\eta)\big)
	\big((A(u) + \eta A^1) + \eps A^0(u+\eta) - \eps A^2(u)\big) \\
	&\phantom{xx}{}- \eta\big(H(u+\eta)+\eps H^0(u+\eta)\big)A^1	
	=: K^1+\ldots+K^5, \nonumber
\end{align}
where
\begin{align*}
	K^1 &= H(u+\eta)(A(u)+\eta A^1), \\
	K^2 &= \eps H^0(u+\eta)(A(u)+\eta A^1) + \eps H(u+\eta)A^0(u+\eta), \\
	K^3 &= \eps^2 H^0(u+\eta)A^0(u+\eta), \\
	K^4 &= -\eta\big(H(u+\eta)+\eps H^0(u+\eta)\big)A^1, \\
	K^5 &= -\eps\big(H(u+\eta) + \eps H^0(u+\eta)\big)A^2(u).
\end{align*}
In the following, let $z\in\R^n$ be fixed.

{\em Step 1: Estimate of $z^TK^1z$.} Since
$H_{ij}(u+\eta)=\delta_{ij}\pi_i(u_i+\eta)^{-2}$ and
\begin{align*}
  A_{ij}(u)+\eta A^1_{ij}
	&= \delta_{ij}a_{i0} + \delta_{ij}\sum_{k=1}^n a_{ik}(u_k+\eta)
	+ a_{ij}(u_i+\eta\delta_{ij}) \\
	&= \delta_{ij}a_{i0} + \delta_{ij}\sum_{k=1}^n a_{ik}(u_k+\eta)
	+ \delta_{ij}a_{ij}(u_i+\eta) + a_{ij}(1-\delta_{ij})u_i,
\end{align*}
we obtain
$$
  K^1_{ij} = \delta_{ij}\frac{\pi_i a_{i0}}{(u_i+\eta)^2}
	+ \delta_{ij}\sum_{k=1}^n \pi_i a_{ik}\frac{u_k+\eta}{(u_i+\eta)^2}
	+ \delta_{ij}\frac{\pi_i a_{ij}}{u_i+\eta}
	+ (1-\delta_{ij})\frac{\pi_i a_{ij}u_i}{(u_i+\eta)^2}
$$
and
\begin{align}\label{4.K1}
  z^T K^1 z &= \sum_{i=1}^n\frac{\pi_i a_{i0}}{(u_i+\eta)^2} z_i^2
	+ 2\sum_{i=1}^n\frac{\pi_i a_{ii}}{u_i+\eta}z_i^2
	+ \sum_{i,k=1,\,i\neq k}^n \pi_i a_{ik}\frac{u_k+\eta}{(u_i+\eta)^2}z_i^2 \\
	&\phantom{xx}{}+ \sum_{i,j=1,\,i\neq j}^n \frac{\pi_i a_{ij}u_i}{(u_i+\eta)^2}z_iz_j.
	\nonumber
\end{align}
We estimate the last term by Young's inequality:
\begin{align*}
  \sum_{i,j=1,\,i\neq j}^n &\frac{\pi_i a_{ij}u_i}{(u_i+\eta)^2}z_iz_j
	\ge -\sum_{i,j=1,\,i\neq j}^n \frac{\pi_i a_{ij}}{u_i+\eta}|z_iz_j| \\
	&\ge -\sum_{i,j=1,\,i\neq j}^n \pi_i a_{ij}\frac{u_j+\eta}{(u_i+\eta)^2}z_i^2
	- \frac14\sum_{i,j=1,\,i\neq j}^n\frac{\pi_i a_{ij}}{u_j+\eta}z_j^2 \\
	&= -\sum_{i,j=1,\,i\neq j}^n \pi_i a_{ij}\frac{u_j+\eta}{(u_i+\eta)^2}z_i^2
	- \frac14\sum_{i,j=1,\,i\neq j}^n\frac{\pi_j a_{ji}}{u_i+\eta}z_i^2.
\end{align*}
The first term on the right-hand side cancels with the third term on the right-hand
side of \eqref{4.K1}. Therefore,
\begin{align*}
  z^TK^1z &\ge \sum_{i=1}^n\frac{\pi_i a_{i0}}{(u_i+\eta)^2} z_i^2
	+ \frac14\sum_{i=1}^n\bigg(8\pi_i a_{ii} - \sum_{j=1,\,j\neq i}^n\pi_j a_{ji}\bigg)
	\frac{z_i^2}{u_i+\eta} \\
  &\ge \sum_{i=1}^n\frac{\pi_i a_{i0}}{(u_i+\eta)^2} z_i^2
	+ \frac{\kappa}{4}\sum_{i=1}^n\frac{z_i^2}{u_i+\eta},
\end{align*}
where $\kappa$ is defined in \eqref{1.aii}.

{\em Step 2: Estimate of $z^TK^2z$.} It follows from
$$
  \eps^{-1}K^2_{ij} = \delta_{ij}\frac{a_{i0}}{u_i+\eta}
	+ \delta_{ij}\sum_{k=1}^n a_{ik}\frac{u_k+\eta}{u_i+\eta} + \delta_{ij}a_{ij}
	+ (1-\delta_{ij})\frac{a_{ij}u_i}{u_i+\eta} + \delta_{ij}\mu_i
$$
that
\begin{align*}
  z^T(\eps^{-1}K^2)z 
	&= \sum_{i=1}^n\bigg(\frac{a_{i0}}{u_i+\eta} + \sum_{k=1,\,k\neq i}^n
  a_{ik}\frac{u_k+\eta}{u_i+\eta} + 2a_{ii} + \mu_i\bigg)z_i^2
	+ \sum_{i,j=1}^n(1-\delta_{ij})\frac{a_{ij}u_i}{u_i+\eta}z_iz_j.
\end{align*}
Using Young's inequality $z_iz_j\ge -(z_i^2+z_j^2)/2$ and taking into account our
choice of $\mu_i$ in Section \ref{sec.aux}, we find that
\begin{align*}
  \sum_{i=1}^n& \mu_i z_i^2
	+ \sum_{i,j=1}^n(1-\delta_{ij})\frac{a_{ij}u_i}{u_i+\eta}z_iz_j
	\ge \sum_{i=1}^n\mu_i z_i^2
	- \frac12\sum_{i=1}^n\sum_{j=1,\,j\neq i}^n a_{ij}(z_i^2+z_j^2) \\
	&\ge \sum_{i=1}^n\bigg(\mu_i
	- \frac12\sum_{j=1,\,j\neq i}^n(a_{ij}+a_{ji})\bigg)z_i^2
	\ge 0.
\end{align*}
This shows that
$$
  z^T K^2 z \ge 2\eps\sum_{i=1}^n a_{ii}z_i^2 \ge 0.
$$

{\em Step 3: Computation of $K^3$, $K^4$, and $K^5$.}
The definitions of the matrices yield
\begin{align*}
  z^T K^3z &= \eps^2\sum_{i=1}^n \frac{\mu_i}{\pi_i}(u_i+\eta)z_i^2 \ge 0, \\
  z^T K^4z &= -\eta\sum_{i=1}^n\bigg(\sum_{k=1}^n a_{ik}+a_{ii}\bigg)
	\bigg(\frac{\pi_i}{(u_i+\eta)^2} + \frac{\eps}{u_i+\eta}\bigg)z_i^2 \\
	&= -\eta\sum_{i,j=1}^n\pi_i a_{ij}\frac{(1+\delta_{ij})z_i^2}{(u_i+\eta)^2}
	- \eta\eps\sum_{i,j=1}^n a_{ij}\frac{(1+\delta_{ij})z_i^2}{u_i+\eta}, \\
  z^T K^5z &= -\eps\eta\sum_{i=1}^n\frac{\mu_i}{\pi_i}
	\bigg(\pi_i\frac{2u_i+\eta}{(u_i+\eta)^2}
	+ \eps\frac{2u_i+\eta}{u_i+\eta}\bigg)z_i^2 \\
	&\ge -2\eta\eps\sum_{i=1}^n\mu_i
	\bigg(\frac{1}{u_i+\eta}+\frac{\eps}{\pi_i}\bigg)z_i^2.
\end{align*}

{\em Step 4: End of the proof.} We insert the estimates for $K^1,\ldots,K^5$
into \eqref{4.HepsAeps}:
\begin{align*}
  z^T & H_\eps(u+\eta)A_\eps(u+\eta)z
	\ge \sum_{i=1}^n \pi_i\bigg(a_{i0} - \eta \sum_{j=1}^n a_{ij}(1+\delta_{ij})\bigg)
	\frac{z_i^2}{(u_i+\eta)^2} \\
	&\phantom{xx}{}+ \frac{\kappa}{4}\sum_{i=1}^n\frac{z_i^2}{u_i+\eta}
	- \eta\eps\sum_{i=1}^n\bigg(\sum_{j=1}^n a_{ij}(1+\delta_{ij})
	+ 2\mu_i\bigg)\frac{z_i^2}{u_i+\eta}
	- 2\eta\eps^2\sum_{i=1}^n\frac{\mu_i}{\pi_i}z_i^2.
\end{align*}
Choosing
$$
  0<\eta\le\eta_0
	:=\min_{i=1,\ldots,n}a_{i0}\bigg(\sum_{j=1}^n a_{ij}(1+\delta_{ij})\bigg)^{-1},
$$
the first term on the left-hand side is nonnegative, and we obtain 
$$
  z^T H_\eps(u+\eta)A_\eps(u)z \ge \frac{\kappa}{4}\sum_{i=1}^n\frac{z_i^2}{u_i+\eta}
	- \eta\eps C_1\sum_{i=1}^n\frac{z_i^2}{u_i+\eta}
	- \eta\eps^2 C_2\sum_{i=1}^n z_i^2,
$$
where $C_1=2\max_{i=1,\ldots,n}(\sum_{j=1}^n a_{ij}+\mu_i)$ and
$C_2=2\max_{i=1,\ldots,n}(\mu_i/\pi_i)$. This finishes the proof.
\end{proof}

\begin{lemma}[Entropy inequality II]\label{lem.ei.eta}
Let $d=1$ and let $0<\eta\le\eta_0$ (see Lemma \ref{lem.HAetaeps}). Then there
exists $C>0$ independent of $\eta$ such that for $0\le s<t$,
$$
  \int_\Omega h(u(t)+\eta)dx + C\sum_{i=1}^n\int_s^t\int_\Omega
	|\na\sqrt{u_i+\eta}|^2 dxd\sigma \le \int_\Omega h(u(s)+\eta)dx.
$$
\end{lemma}

\begin{proof}
We use $v^k=(v_1^k,\ldots,v_n^k)$ with
$$
  v_i^k = \frac{\pa h_\eps}{\pa u_i}(u^k+\eta)
	= \pi_i\bigg(1-\frac{1}{u_i^k+\eta}\bigg) + \eps\log(u_i^k+\eta)
$$
as a test function in the weak formulation of the approximate equations \eqref{3.eps}:
\begin{align*}
  \frac{1}{\tau}\int_\Omega & (u^k-u^{k-1})\cdot v^k dx
	+ \int_\Omega \na v^k:B_\eps(w^k)\na w^k dx \\
	&{}+ \eps\int_\Omega\bigg(\sum_{|\alpha|=1}D^\alpha w^k\cdot D^\alpha v^k
	+ w^k\cdot v^k\bigg)dx = 0.
\end{align*}
Note that we have chosen $\delta=\eps$; see Remark \ref{rem}.
The convexity of $h_\eps$ implies that
$$
  (u^k-u^{k-1})\cdot v^k = ((u^k+\eta)-(u^{k-1}+\eta))\cdot h'_\eps(u^k+\eta)
	\ge h_\eps(u^k+\eta) - h_\eps(u^{k-1}+\eta).
$$
Furthermore, by the definition of $v^k$,
$$
  \sum_{|\alpha|=1}D^\alpha w^k\cdot D^\alpha v^k
	= \na w^k\cdot\na v^k = \bigg(\frac{\pi_i}{(u_i^k)^2} + \frac{\eps}{u_i^k}\bigg)
	\bigg(\frac{\pi_i}{(u_i^k+\eta)^2}+\frac{\eps}{u_i^k+\eta}\bigg)|\na u_i^k|^2
	\ge 0.
$$
Note that $u_i^k>0$, so quotients of the type $\pi_i/(u_i^k)^2$ are well-defined.
It follows from $\na v^k = H_\eps(u^k+\eta)\na u^k$ and Lemma \ref{lem.HAetaeps}
that
\begin{align*}
  \na & v^k:B_\eps(w^k)\na w^k = \na u^k:H_\eps(u^k+\eta)A_\eps(u^k)\na u^k \\
	&\ge \frac{\kappa}{4}\sum_{i=1}^n\frac{|\na(u_i^k+\eta)|^2}{u_i^k+\eta}
	- \eta\eps C_1\sum_{i=1}^n\frac{|\na u_i^k|^2}{u_i^k+\eta}
	- \eta\eps^2 C_2\sum_{i=1}^n|\na u_i^k|^2.
\end{align*}
Summarizing, this gives
\begin{align*}
  \int_\Omega h_\eps(u^k+\eta)dx &+ \kappa\tau\sum_{i=1}^n|\na(u_i^k+\eta)^{1/2}|^2 dx
	\le \int_\Omega h_\eps(u^{k-1}+\eta)dx \\
	&{}-\eps\tau\int_\Omega w^k\cdot v^k dx
	+ \eta\eps\tau C_1\sum_{i=1}^n\int_\Omega\frac{|\na u_i^k|^2}{u_i^k+\eta}
	+ \eta\eps^2\tau C_2\sum_{i=1}^n|\na u_i^k|^2.
\end{align*}
We sum this inequality from $k=j,\ldots,\ell$ for $j<\ell$:
\begin{align}\label{4.aux2}
  \int_\Omega & h_\eps(u^\ell+\eta)dx + \kappa\sum_{i=1}^n\sum_{k=j}^\ell\tau\int_\Omega
	|\na(u_i^k+\eta)^{1/2}|^2 dx \\
	&\le \int_\Omega h_\eps(u^{j-1}+\eta)dx
	-\eps\sum_{k=j}^\ell\tau\int_\Omega w^k\cdot v^k dx \nonumber \\
	&\phantom{xx}{}+\eta\eps C_1\sum_{i=1}^n\sum_{k=j}^\ell\tau\int_\Omega
	\frac{|\na u_i^k|^2}{u_i^k+\eta}dx
	+ \eta\eps^2 C_2\sum_{i=1}^n\sum_{k=j}^\ell\tau\int_\Omega|\na u_i^k|^2 dx.
	\nonumber
\end{align}
We know from Lemma \ref{lem.est1} that
$$
  \sum_{k=0}^N\tau\|(u_i^k)^{1/2}\|_{H^1(\Omega)}^2
	+ \eps\sum_{k=0}^N\tau\|u_i^k\|_{H^1(\Omega)}^2 \le C.
$$
Since $|\na u_i^k|^2/(u_i^k+\eta)=4u_i^k|\na(u_i^k)^{1/2}|^2/(u_i^k+\eta)
\le 4|\na(u_i^k)^{1/2}|$,
the last two terms on the right-hand side of \eqref{4.aux2} are bounded
from above by $\eta\eps C$. Thus, it remains to estimate the first term on the
right-hand side of \eqref{4.aux2}. We write
\begin{align*}
  & -\eps\sum_{k=j}^\ell\tau\int_\Omega w^k\cdot v^k dx = I_1 + I_2,
	\quad\mbox{where} \\
	& I_1 := -\eps\sum_{i=1}^n\sum_{k=j}^\ell\tau\int_\Omega
	\pi_i\bigg(1-\frac{1}{u_i^k+\eta}\bigg)w_i^k dx, \\
	& I_2 := -\eps^2\sum_{i=1}^n\sum_{k=j}^\ell\tau\int_\Omega
	\log(u_i^k+\eta)w_i^k dx.
\end{align*}
Since estimate \eqref{3.deis} shows that
\begin{equation}\label{4.w}
  \eps\sum_{i=1}^n\sum_{k=1}^N\tau\|w_i^k\|_{H^1(\Omega)}^2 \le C,
\end{equation}
we obtain
$$
  I_1 \le \eps\sum_{i=1}^n\sum_{k=j}^\ell\tau\int_\Omega
	\pi_i\bigg(1+\frac{1}{\eta}\bigg)|w_i^k| dx \le C(\eta,T)\sqrt{\eps}.
$$
To estimate $I_2$, we first compute
\begin{align*}
  \int_\Omega & |\log(u_i^k+\eta)|dx \le \int_\Omega|\log(u_i^k+\eta)-\log\eta|dx
	+ \int_\Omega|\log\eta|dx \\
	&= \int_\Omega\bigg|u_i^k\int_0^1\frac{d\theta}{\theta u_i^k+\eta}\bigg|dx
	+ |\log\eta|\operatorname{meas}(\Omega)
	\le \frac{1}{\eta}\|u_i^k\|_{L^1(\Omega)} + |\log\eta|\operatorname{meas}(\Omega).
\end{align*}
By Lemma \ref{lem.est1}, $\|u_i^k\|_{L^1(\Omega)}$ is bounded uniformly in $k$
(and $(\eps,\tau)$). We conclude from \eqref{4.w} and the continuous
embedding $H^1(\Omega)\hookrightarrow L^\infty(\Omega)$ in one space dimension
that
\begin{align*}
  I_2 &\le \eps^2\sum_{i=1}^n\sum_{k=j}^\ell\tau\|\log(u_i^k+\eta)\|_{L^1(\Omega)}
	\|w_i^k\|_{L^\infty(\Omega)} \\
	&\le \eps^2 C\sum_{i=1}^n\max_{k=1,\ldots,N}
	\|\log(u_i^k+\eta)\|_{L^1(\Omega)}
	\sqrt{T}\sum_{k=1}^N\tau\|w_i^k\|_{H^1(\Omega)}^2 \le C(\eta,T)\eps^{3/2}.
\end{align*}

Summarizing these estimates, we infer from \eqref{4.aux2}, using the notation from
Section \ref{sec.approx}, that
\begin{align}\label{4.aux3}
  \int_\Omega h_\eps & (u^{(\tau)}(t)+\eta)dx + \kappa\sum_{i=1}^n\int_s^t\int_\Omega
	|\na (u_i^{(\tau)}+\eta)^{1/2}|^2 dx d\sigma \\
	&\le \int_\Omega h_\eps(u^{(\tau)}(s-\tau)+\eta)dx
	+ \eta\eps C + C(\eta,T)\sqrt{\eps}(1+\eps), \nonumber
\end{align}
where $s\in((j-1)\tau,j\tau]$, $t\in((\ell-1)\tau,\ell\tau]$.
Since 
$$
  |\na (u_i^{(\tau)}+\eta)^{1/2}| 
	= \bigg|\frac{\na u_i^{(\tau)}}{2(u_i^{(\tau)}+\eta)^{1/2}}\bigg|
  \le \bigg|\frac{\na u_i^{(\tau)}}{2(u_i^{(\tau)})^{1/2}}\bigg|
	= |\na (u_i^{(\tau)})^{1/2}|,
$$
it follows from estimate \eqref{3.deis} that 
$\sum_{i=1}^n\big\|\na (u_i^{(\tau)}+\eta)^{1/2}\big\|_{L^2(Q_T)} \le C$.
We have already proved that, up to a subsequence, $u_i^{(\tau)}\to u_i$ strongly
in $L^{3/2}(Q_T)$ as $(\eps,\tau)\to 0$. We infer that
$\na(u_i^{(\tau)}+\eta)^{1/2}\rightharpoonup\na (u_i+\eta)^{1/2}$
weakly in $L^2(Q_T)$.
Therefore, $u_i^{(\tau)}(t)\to u_i(t)$ strongly in
$L^{3/2}(\Omega)$ for a.e.\ $t\in(0,T)$ and
$$
  \int_\Omega\log(u_i^{(\tau)}(t)+\eta)dx \to \int_\Omega\log(u_i(t)+\eta)dx.
$$
By the weak lower semicontinuity of the norm,
$$
  \int_s^t\int_\Omega|\na\sqrt{u_i+\eta}|^2 dxd\sigma
	\le \liminf_{(\tau,\eps)\to 0}\int_s^t|\na(u_i^{(\tau)}+\eta)^{1/2}|^2 dx.
$$
The limit $(\tau,\eps)\to 0$ in \eqref{4.aux3} concludes the proof.
\end{proof}

Next, we introduce for $0<\eta\le\eta_0$ the relative entropy
\begin{align*}
  \H_\eta(u|\bar{u}) &= \int_\Omega\big(h(u+\eta)-h(\bar{u}+\eta)
	-h'(\bar{u}+\eta)\cdot((u+\eta)-(\bar{u}+\eta))\big)dx \\
	&= \sum_{i=1}^n\pi_i\int_\Omega\bigg(\frac{u_i+\eta}{\bar{u}_i+\eta}
	-\log\frac{u_i+\eta}{\bar{u}_i+\eta}-1\bigg)dx.
\end{align*}
Because of mass conservation, we have
$\int_\Omega((u_i+\eta)/(\bar{u}_i+\eta)-1)dx=0$, implying that
\begin{equation}\label{4.defHeta}
  \H_\eta(u|\bar{u}) = \sum_{i=1}^n\pi_i
	\int_\Omega\big(\log(\bar{u}_i+\eta)-\log (u_i+\eta)\big)dx.
\end{equation}
In view of Lemma \ref{lem.ei.eta}, we can formulate the relative entropy inequality as 
\begin{equation}\label{4.Heta}
  \H_\eta(u(t)|\bar{u}) + C\sum_{i=1}^n\int_s^t\int_\Omega|\na\sqrt{u_i+\eta}|^2dx
	\le \H_\eta(u(s)|\bar{u}), \quad 0<s<t.
\end{equation}
We claim that the relative entropy decays to zero as $t\rightarrow\infty$. To prove
this, we need some preparation.

\begin{lemma}\label{lem.poincare}
Let $g\in L^\infty(0,\infty;L^1(\Omega))$ with $g\ge 0$
and $\na\sqrt{g}\in L^2(0,\infty;L^2(\Omega))$ be such that
$\bar{g}:=\fint_\Omega g(x,t)dx$ is independent
of $t>0$ (i.e., $g$ conserves the mass).
Then there exists a constant $C>0$ independent of $g$ such that for $t>0$,
$$
  \|\sqrt{g(t)}-\sqrt{\bar{g}}\|_{L^2(\Omega)}
	\le C\|\na \sqrt{g(t)}\|_{L^2(\Omega)}.
$$
\end{lemma}

\begin{proof}
The proof is similar to that one in \cite[Lemma 7]{CGZ22} but some arguments are
different. We argue by contradiction.
Assume that there exists a sequence $(t_n)_{n\in\N}$ such that
\begin{equation}\label{4.contra}
  n\|\na \sqrt{g(t_n)}\|_{L^2(\Omega)}
	< \|\sqrt{g(t_n)}-\sqrt{\bar{g}}\|_{L^2(\Omega)}\quad\mbox{for all }n\in\N.
\end{equation}
This implies that $\|\sqrt{g(t_n)}-\sqrt{\bar{g}}\|_{L^2(\Omega)}>0$ and we can
define
$$
  v_n := \frac{\sqrt{g(t_n)}-\sqrt{\bar{g}}}{\|\sqrt{g(t_n)}-\sqrt{\bar{g}}
	\|_{L^2(\Omega)}}, \quad n\in\N.
$$
It follows from \eqref{4.contra} that
$$
  \|\na v_n\|_{L^2(\Omega)} = \frac{\|\na \sqrt{g(t_n)}\|_{L^2(\Omega)}}{
	\|\sqrt{g(t_n)}-\sqrt{\bar{g}}\|_{L^2(\Omega)}} < \frac{1}{n},
$$
such that $\na v_n\to 0$ strongly in $L^2(\Omega)$ as $n\to\infty$. By definition,
$\|v_n\|_{L^2(\Omega)}=1$ for all $n\in\N$, i.e., $(v_n)$ is bounded in $H^1(\Omega)$.
Taking into account the compact embedding $H^1(\Omega)\hookrightarrow L^2(\Omega)$,
there exists a subsequence, which is not relabeled, such that
$v_n\to v$ strongly in $L^2(\Omega)$ and $v_n\rightharpoonup v$ weakly in
$H^1(\Omega)$ as $n\to\infty$.
We deduce from $\na v_n\to 0$ strongly in $L^2(\Omega)$ that
$v$ is a constant and, because of $\|v_n\|_{L^2(\Omega)}=1$, we have $v\neq 0$.

Now, we show that $\|\sqrt{g(t_n)}-\sqrt{\bar{g}}\|_{L^2(\Omega)}\to 0$
as $n\to\infty$. Otherwise, by contradiction, there exists a subsequence of
$(g(t_n))_{n\in\N}$ (not relabeled) and $c>0$ such that
$\|\sqrt{g(t_n)}-\sqrt{\bar{g}}\|_{L^2(\Omega)}\ge c$ for every $n\in\N$.
Because of
\begin{equation}\label{4.convv}
  \frac{\sqrt{g(t_n)}-\sqrt{\bar{g}}}{\|\sqrt{g(t_n)}-\sqrt{\bar{g}}
	\|_{L^2(\Omega)}} \to v\quad\mbox{strongly in }L^2(\Omega)\mbox{ and a.e. in }\Omega,
\end{equation}
Egorov's theorem \cite[Theorem 4.29]{Bre11} shows that, for any $\eps>0$,
there exists $\Omega_\eps\subset\Omega$ such that $\operatorname{meas}(\Omega
\setminus\Omega_\eps)<\eps$ and
$$
  \frac{\sqrt{g(t_n)}-\sqrt{\bar{g}}}{\|\sqrt{g(t_n)}-\sqrt{\bar{g}}
	\|_{L^2(\Omega)}}\to v\quad\mbox{strongly in }L^\infty(\Omega_\eps).
$$
Since $v$ is a nonzero constant, there exist $c>0$ and $N_\eps\in\N$ such that for all
$n>N_\eps$,
\begin{align*}
  \sqrt{g(t_n)} \ge \sqrt{\bar{g}} + \frac{c}{2}v
	&\quad\mbox{a.e. in }\Omega_\eps\mbox{ if }v>0, \\
  \sqrt{g(t_n)} \le \sqrt{\bar{g}_i} - \frac{c}{2}(-v)
	&\quad\mbox{a.e. in }\Omega_\eps\mbox{ if }v<0.
\end{align*}
Thus, there exist $K_1$, $K_2>0$ independent of $\eps$ such that in $\Omega_\eps$,
$g(t_n)\ge \bar{g}+K_1$ if $v>0$ and $g(t_n)\le \bar{g}-K_2$ if $v<0$.
As the integral is absolutely continuous and $\eps>0$ is arbitrary, this
contradicts the constraint $\fint g(t_n)dx=\bar{g}$. We infer that
$$
  \|\sqrt{g(t_n)}-\sqrt{\bar{g}}\|_{L^2(\Omega)}\to 0\quad\mbox{as }n\to\infty
$$
and consequently,
$$
  \frac{g(t_n)-\bar{g}}{\sqrt{g(t_n)}-\sqrt{\bar{g}}}
	= \sqrt{g(t_n)}+\sqrt{\bar{g}} \to 2\sqrt{\bar{g}}
	\quad\mbox{strongly in }L^2(\Omega).
$$
Then the previous result and convergence \eqref{4.convv} imply that
$$
  \frac{g(t_n)-\bar{g}}{\|\sqrt{g(t_n)}-\sqrt{\bar{g}}\|_{L^2(\Omega)}}
	= \frac{g(t_n)-\bar{g}}{\sqrt{g(t_n)}-\sqrt{\bar{g}}}
	\frac{\sqrt{g(t_n)}-\sqrt{\bar{g}}}{\|\sqrt{g(t_n)}-\sqrt{\bar{g}}\|_{L^2(\Omega)}}
	\to 2\sqrt{\bar{g}}v
$$
strongly in $L^1(\Omega)$. However, this gives
$$
  \int_\Omega\frac{g(t_n)-\bar{g}}{\|\sqrt{g(t_n)}-\sqrt{\bar{g}}\|_{L^2(\Omega)}}
	dx \to \int_\Omega 2\sqrt{\bar{g}}v dx\neq 0,
$$
which violates the conservation of mass and ends the proof.
\end{proof}

The previous lemma and the entropy inequality in Lemma \ref{lem.ei} imply that
$$
  \int_0^\infty\|\sqrt{u_i(t)}-\sqrt{\bar{u}_i}\|_{L^2(\Omega)}^2 dt
	\le C\int_0^\infty\|\na \sqrt{u_i(t)}\|_{L^2(\Omega)}^2 dt \le C(u^0).
$$
Consequently, there exists a sequence $(t_k)\subset[0,\infty)$ satisfying
$t_k\to\infty$ as $k\to\infty$ such that 
$$
  \lim_{k\to\infty}\|\sqrt{u_i(t_k)}-\sqrt{\bar{u}_i}\|_{L^2(\Omega)} = 0.
$$
This shows that
\begin{align*}
  \lim_{k\to\infty}\|u_i(t_k)-\bar{u}_i\|_{L^1(\Omega)}
	&\le \lim_{k\to\infty}\|\sqrt{u_i(t_k)}+\sqrt{\bar{u}_i}\|_{L^2(\Omega)}
	\|\sqrt{u_i(t_k)}-\sqrt{\bar{u}_i}\|_{L^2(\Omega)} \\
	&\le C(u^0)\lim_{k\to\infty}\|\sqrt{u_i(t_k)}-\sqrt{\bar{u}_i}\|_{L^2(\Omega)} = 0.
\end{align*}
In particular, we obtain, for any fixed $\eta>0$,
$$
  \lim_{k\to\infty}\int_\Omega\log(u_i(t_k)+\eta)dx
	= \int_\Omega\log(\bar{u}_i+\eta)dx,
$$
and in view of definition \eqref{4.defHeta} of the relative entropy, this implies that
$$
  \lim_{k\to\infty}\H_\eta(u(t_k)|\bar{u}) = 0.
$$
Since $t\mapsto\H_\eta(u(t)|\bar{u})$ is bounded and nonincreasing by \eqref{4.Heta},
the convergence holds for all sequences $t\to\infty$:
$$
  \lim_{t\to\infty}\H_\eta(u(t)|\bar{u}) = 0.
$$
Finally, by the Csisz\'ar--Kullback inequality (see Proposition \ref{lem.ck} in the
appendix),
$$
  \lim_{t\to\infty}\|u_i(t)-\bar{u}_i\|_{L^1(\Omega)}
	\le C\|\bar{u}_i+\eta\|_{L^2(\Omega)}
  \lim_{t\to\infty}\H_\eta(u(t)|\bar{u})^{1/2} = 0
$$
for all $0<\eta\le\eta_0$, which ends the proof.


\begin{appendix}
\section{Auxiliary results}\label{sec.app}

\begin{lemma}\label{lem.equi}
Let $n=3$, $a_{13}=a_{21}=a_{32}=1$, and $a_{12}=a_{23}=a_{31}=0$.
Then there exist $\pi_1,\pi_2,\pi_3>0$ satisfying $\kappa>0$
(see \eqref{1.aii}) if and only if $a_{11}a_{22}a_{33}>8^{-3}$.
\end{lemma}

\begin{proof}
The condition $\kappa>0$ is equivalent to
$8\pi_1a_{11}>\pi_2$, $8\pi_2a_{22}>\pi_3$, and
$8\pi_3a_{33}>\pi_1$. Multiplying these inequalities immediately gives
$8^3a_{11}a_{22}a_{33}>1$. On the other hand, if this inequality is satisfied,
we set
$$
  \pi_1 = 1, \quad \pi_2 = \frac12\bigg(8a_{11}+\frac{1}{8^2a_{22}a_{33}}\bigg),
	\quad \pi_3 = \frac12\bigg(8\pi_2 a_{22}+\frac{1}{8a_{33}}\bigg).
$$
Then $8\pi_1a_{11}>\pi_2$ is equivalent to $8^3a_{11}a_{22}a_{33}>1$, and
both $8\pi_2a_{22}>\pi_3$ and $8\pi_3a_{33}>\pi_1$ are equivalent to
$8^2\pi_2a_{22}a_{33}>1$, which, by definition of $\pi_2$, is equivalent to
$8^3a_{11}a_{22}a_{33}>1$ again.
\end{proof}

The following result is proved in \cite[Section 4.3, page 71, example (c)]{CJMTU01}.

\begin{proposition}[Csisz\'ar--Kullback inequality]\label{lem.ck}
Let $\Omega\subset\R^d$ be a domain and $u\in L^1(\Omega)$.
We set $\bar{u}=\fint_\Omega udx$ and
$\H(u|\bar{u}) = \int_\Omega(\log(\bar{u}+\eta)-\log(u+\eta))dx$. Then
$$
  \|u-\bar{u}\|_{L^1(\Omega)} \le \sqrt{8}\|\bar{u}\|_{L^2(\Omega)}
	\H(u|\bar{u})^{1/2}.
$$
\end{proposition}

\end{appendix}


\section*{Data availability statement}

Data sharing not applicable to this article as no datasets were generated or 
analyzed during the current study.


\section*{Conflict of interest statement}

The authors have no competing interests to declare that are relevant to the 
content of this article.



\begin{thebibliography}{11}
\bibitem{Ama89} H.~Amann. Dynamic theory of quasilinear parabolic systems. III.
Global existence. {\em Math. Z.} 202 (1989), 219--250.

\bibitem{Bre11} H.~Br\'ezis. {\em Functional Analysis, Sobolev Spaces and
Partial Differential Equations}. Springer, New York, 2011.

\bibitem{CJMTU01} J.~A.~Carrillo, A.~J\"ungel, P.~Markowich, G.~Toscani, and
A.~Unterreiter. Entropy dissipation methods for degenerate parabolic problems and
generalized Sobolev inequalities. {\em Monatsh. Math.} 133 (2001), 1--82.

\bibitem{CGZ22} L.~Chen, S.~G\"ottlich, and N.~Zamponi. Bounded weak solution and
long time behavior of a degenerate particle flow model.
Submitted for publication, 2022. arXiv:2202.04416.

\bibitem{ChJu04} L.~Chen and A.~J\"ungel. Analysis of a multi-dimensional parabolic
population model with strong cross-diffusion.
{\em SIAM J. Math. Anal}. 36 (2004), 301--322.

\bibitem{ChJu06} L.~Chen and A.~J\"ungel. Analysis of a parabolic cross-diffusion
population model without self-diffusion. {\em J. Diff. Eqs.} 224 (2006), 39--59.

\bibitem{CDJ18} X.~Chen, E.~Daus, and A.~J\"ungel. Global existence analysis of
cross-diffusion population systems for multiple species.
{\em Arch. Ration. Mech. Anal.} 227 (2018), 715--747.

\bibitem{DLM14} L.~Desvillettes, T.~Lepoutre, and A.~Moussa. Entropy, duality, and
cross diffusion. {\em SIAM J. Math. Anal.} 46 (2014), 820--853.

\bibitem{DLMT15} L.~Desvillettes, T.~Lepoutre, A.~Moussa, and A.~Trescases.
On the entropic structure of reaction-cross diffusion systems.
{\em Commun. Partial Diff. Eqs.} 40 (2015), 1705--1747.

\bibitem{Dre08} M.~Dreher. Analysis of a population model with strong cross-diffusion
in unbounded domains. {\em Proc. Roy. Soc. Edinb. Sec. A} 138 (2008), 769--786.

\bibitem{DrJu12} M.~Dreher and A.~J\"ungel. Compact families of piecewise constant
functions in $L^p(0,T;B)$. {\em Nonlin. Anal.} 75 (2012), 3072--3077.

\bibitem{Jue15} A.~J\"ungel. The boundedness-by-entropy method for cross-diffusion
systems. {\em Nonlinearity} 28 (2015), 1963--2001.

\bibitem{Jue16} A.~J\"ungel. {\em Entropy Methods for Diffusive Partial Differential
Equations}. Springer Briefs Math., Springer, 2016.

\bibitem{JuMa06} A.~J\"ungel and D.~Matthes. An algorithmic construction of entropies
in higher-order nonlinear PDEs. {\em Nonlinearity} 19 (2006), 633--659.

\bibitem{JuZa16} A.~J\"ungel and N.~Zamponi. Qualitative behavior of solutions to
cross-diffusion systems from population dynamics.
{\em J. Math. Anal. Appl.} 440 (2016), 794--809.

\bibitem{LeMo17} T.~Lepoutre and A.~Moussa. Entropic structure and duality for multiple
species cross-diffusion systems. {\em Nonlin. Anal.} 159 (2017), 298--315.

\bibitem{LMN00} Y.~Lou, S.~Mart\'{\i}nez, and W.-M.~Ni. On $3\times 3$ Lotka--Volterra
competition systems with cross-diffusion. {\em Discrete Contin. Dynam. Sys.} 6
(2000), 175--190.

\bibitem{Rou05} T.~Roub\'{\i}\v{c}ek. {\em Nonlinear Partial Differential
Equations with Applications}. Birkh\"auser, Basel, 2005.

\bibitem{SKT79} N.~Shigesada, K.~Kawasaki, and E.~Teramoto. Spatial segregation
of interacting species. {\em J. Theor. Biol.} 79 (1979), 83--99.

\bibitem{WeFu09} Z.~Wen and S.~Fu. Global solutions to a class of multi-species
reaction-diffusion systems with cross-diffusions arising in population dynamics.
{\em J. Comput. Appl. Math.} 230 (2009), 34--43.

\end{thebibliography}
\end{document}